\documentclass[a4paper,10pt]{amsart}
\usepackage{enumerate}
\usepackage{amsthm,amsmath,amssymb,graphics,graphicx,hyperref,epstopdf,mathrsfs}
\usepackage{breqn}           
\title{The order of dominance of a monomial ideal}
\author{Guillermo Alesandroni}
\address{2000 Rosario, Santa Fe, Argentina}
\email{guillea@okstate.edu, alesandronig@yahoo.com}

\newtheorem{theorem}{Theorem}[section]

\newtheorem{corollary}[theorem]{Corollary}
\newtheorem{lemma}[theorem]{Lemma}

\theoremstyle{definition}
\newtheorem{definition}[theorem]{Definition}

\newtheorem{example}[theorem]{Example}
\newtheorem{construction}[theorem]{Construction}

\DeclareMathOperator{\betti}{b}
\DeclareMathOperator{\pd}{pd}
\DeclareMathOperator{\mdeg}{mdeg}
\DeclareMathOperator{\lcm}{lcm}

\DeclareMathOperator{\Max}{max}

\DeclareMathOperator{\hdeg}{hdeg}

\DeclareMathOperator{\rank}{rank}
\DeclareMathOperator{\codim}{codim}
\DeclareMathOperator{\odom}{odom}
\DeclareMathOperator{\pol}{pol}

\begin{document}
\maketitle
\begin{abstract}
Let $S$ be a polynomial ring in $n$ variables over a field, and consider a monomial ideal $M=(m_1,\ldots,m_q)$ of $S$. We introduce a new invariant, called order of dominance of $S/M$, and denoted $\odom(S/M)$, which has many similarities with the codimension of $S/M$. We use the order of dominance to characterize the class of Scarf ideals that are Cohen-Macaulay, and also to characterize when the Taylor resolution is minimal. In addition, we show that $\odom(S/M)$ has the following properties:
\begin{enumerate}[(i)]
\item $\codim(S/M) \leq \odom(S/M)\leq \pd(S/M)$.
\item $\pd(S/M)=n$ if and only if $\odom(S/M)=n$.
\item $\pd(S/M) = 1$ if and only if $\odom(S/M) = 1$.
\item If $\odom(S/M) = n-1$, then $\pd(S/M) = n-1$.
\item  If $\odom(S/M) = q-1$, then $\pd(S/M) = q-1$.
\end{enumerate}
\end{abstract}

\section{Introduction}
One of the fundamental concepts of dimension theory is that of codimension or height. In the context of monomial ideals this definition is simple and intuitive: if $M$ is a monomial ideal in $S=k[x_1,\ldots,x_n]$, where $k$ is a field, the codimension of the $S$-module $S/M$ is the cardinality of the smallest subset $A$ of $\{x_1,\ldots,x_n\}$ such that every minimal generator of $M$ is divisible by some variable of $A$. The importance of codimension in the study of monomial resolutions is obvious as it tells us how short a resolution can be. That is, $\codim(S/M)\leq \pd(S/M)$.

In this article we introduce a new and sharper lower bound for the projective dimension; a notion that we call order of dominance. Although unorthodox, we define the order of dominance in this introduction after analyzing three motivational examples.

Here is the first example. Let $M_1=(a,b,c)$ be an ideal of $S=k[a,b,c]$. Then, $\pd(S/M_1)\geq \codim(S/M_1)= \#\{a,b,c\}=3$. In this case, $\codim(S/M_1)$ is a good lower bound of $\pd(S/M_1)$ because, as we know, $\pd(S/M_1)=3$. 

However, we can define a slightly different ideal for which codimension is a bad lower bound. Consider, for instance, the ideal $M_2=(ad,bd,cd)$ of the ring $S=k[a,b,c,d]$. In this example we have that $\codim(S/M_2)=\# \{d\}=1$, while 
$\pd(S/M_2)=3$. This apparently discouraging fact inspired the present work.

Let us revisit the definition of codimension given in the first paragraph. It is easy to see that the subset $A$ of $\{x_1,\ldots,x_n\}$, whose cardinality is precisely $\codim(S/M)$, has the following properties:
\begin{enumerate}[(i)]
\item Every minimal generator of $M$ is divisible by some variable of $A$.
\item If $A'\subsetneq A$, there is a minimal generator of $M$ that is not divisible by any variable of $A'$.
\end{enumerate}  
Indeed, the subset $A$ of $\{x_1,\ldots,x_n\}$ with $\#A=\codim(S/M)$ is the smallest subset of $\{x_1,\ldots,x_n\}$ satisfying (i) and (ii) but, usually, not the only one.
It follows from one of our main results (Theorem \ref{Theorem 3.4}) that the cardinality of any subset of $\{x_1\ldots,x_n\}$ satisfying (i) and (ii) is a lower bound for $\pd(S/M)$. Thus, when we want to find a lower bound for $\pd(S/M)$ we can look for the largest cardinality of a subset of $\{x_1,\ldots,x_n\}$ satisfying (i) and (ii).

In the particular case of $M_2$, $\codim(S/M_2)$ is given by the cardinality of $\{d\}$, and $\{d\}$ certainly satisfies (i) and (ii). However, the set $\{a,b,c\}$ also satisfies (i) and (ii) and thus, according to Theorem \ref{Theorem 3.4}, $\pd(S/M_2)\geq \#\{a,b,c\}=3$. This is good news because, in fact, $\pd(S/M_2)=3$.

Now the scenario looks better but, as the next example shows, it is not good enough. Let $M_3=(ad,bd,cd,d^2)$ be an ideal of $S=k[a,b,c,d]$ (note that $M_3$ is obtained from $M_2$ by adding a minimal generator). In this case, $\codim(S/M_3)=\#\{d\}=1$, and $\pd(S/M_3)=4$. By simple inspection, we can verify that $\{d\}$ is the only subset of $\{a,b,c,d\}$ satisfying (i) and (ii). Thus, we are unable to improve the lower bound given by $\codim(S/M_3)$. Yet, we can do one last trick. 

By polarizing, we can interpret $M_3$ as the monomial ideal $M'_3=(a_1d_1,b_1d_1,c_1d_1,d_1d_2)$ of $S'=k[a_1,b_1,c_1,d_1,d_2]$. It is easy to see that $\{a_1,b_1,c_1,d_2\}$ is the largest subset of $\{a_1,b_1,c_1,d_1,d_2\}$ satisfying (i) and (ii). Thus, we have that $\pd(S/M_3)=\pd(S'/M'_3)\geq \#\{a_1,b_1,c_1,d_2\}=4$. This lower bound is optimal, for $\pd(S/M_3)=4$.

We have just introduced the main concept of this paper; the order of dominance. More precisely, if $M$ is a monomial ideal of $S$, $M_{\pol}$ is the polarization of $M$, and we view $M_{\pol}$ as an ideal of the ring $S_{\pol}$, then the \textbf{order of dominance} of $S/M$, denoted $\odom(S/M)$, is the largest cardinality of a set $A$ of variables of $S_{\pol}$, having the following properties:
\begin{enumerate}[(i)]
\item Every minimal generator of $M_{\pol}$ is divisible by some variable of $A$.
\item If $A'\subsetneq A$, then there is a minimal generator of $M_{\pol}$ that is not divisible by any variable of $A'$.
\end{enumerate}

We will show that the order of dominance can be used to give a careful description of some basis elements (including homological degree and multidegree) of the minimal resolution of $S/M$ (Lemma \ref{lemma 2}). In addition, we will use the order of dominance to characterize the class of Scarf ideals that are Cohen-Macaulay (Corollary \ref{C-M}), and also to characterize when the Taylor resolution is minimal (Theorem \ref{Taylor}). Furthermore, we will characterize all Cohen-Macaulay monomial ideals in 3 variables (Corollary \ref{last}).  Below we describe other properties of  $\odom(S/M)$. Let $M=(m_1,\ldots,m_q)$ be an arbitrary monomial ideal in $S=k[x_1,\ldots,x_n]$. Then:
\begin{enumerate}[(a)]
\item $\codim(S/M) \leq \odom(S/M)\leq \pd(S/M)$, Theorem \ref{nets 6}.
\item  $\pd(S/M)=n$ if and only if $\odom(S/M)=n$, Theorem \ref{nets 5}.
\item $\pd(S/M) = 1$ if and only if $\odom(S/M) = 1$, Theorem \ref{Theorem pd odom}.
\item If $\odom(S/M) = n-1$, then $\pd(S/M) = n-1$, Corollary \ref{cor.} (i).
\item  If $\odom(S/M) = q-1$, then $\pd(S/M) = q-1$, Corollary \ref{cor.} (ii).
\item  If M is Scarf, then $\pd(S/M)=\odom(S/M)$, Theorem \ref{ScarfPD}.
\end{enumerate}

\section{Background and Notation}

Throughout this paper $S$ represents a polynomial ring in $n$ variables over a field $k$. In some examples, $n$ takes a specific value, and  the variables are denoted with the letters $a$, $b$, $c$, etc. In Corollary \ref{last}, $S=k[x_1,x_2, x_3]$.  Everywhere else, $n$ is arbitrary, and $S=k[x_1,\ldots, x_n]$. The letter $M$ always represents a monomial ideal in $S$. 

We open this section by defining the Taylor resolution as a multigraded free resolution, something that will turn out to be fundamental in the present work. The construction that we give below can be found in [Me].
 
\begin{construction}
Let $M=(m_1,\ldots,m_q)$. For every subset $\{m_{i_1},\ldots,m_{i_s}\}$ of $\{m_1,\ldots,m_q\}$, with $1\leq i_1<\cdots<i_s\leq q$, 
we create a formal symbol $[m_{i_1},\ldots,m_{i_s}]$, called a \textbf{Taylor symbol}. The Taylor symbol associated to $\{\}$ will be denoted by $[\varnothing]$.
For each $s=0,\ldots,q$, set $F_s$ equal to the free $S$-module with basis $\{[m_{i_1},\ldots,m_{i_s}]:1\leq i_1<\cdots<i_s\leq q\}$ given by the 
${q\choose s}$ Taylor symbols corresponding to subsets of size $s$. That is, $F_s=\bigoplus\limits_{i_1<\cdots<i_s}S[m_{i_1},\ldots,m_{i_s}]$ 
(note that $F_0=S[\varnothing]$). Define
\[f_0:F_0\rightarrow S/M\]
\[s[\varnothing]\mapsto f_0(s[\varnothing])=s\]
For $s=1,\ldots,q$, let $f_s:F_s\rightarrow F_{s-1}$ be given by
\[f_s\left([m_{i_1},\ldots,m_{i_s}]\right)=
 \sum\limits_{j=1}^s\dfrac{(-1)^{j+1}\lcm(m_{i_1},\ldots,m_{i_s})}{\lcm(m_{i_1},\ldots,\widehat{m_{i_j}},\ldots,m_{i_s})}
 [m_{i_1},\ldots,\widehat{m_{i_j}},\ldots,m_{i_s}]\]
 and extended by linearity.
 The \textbf{Taylor resolution} $\mathbb{T}_M$ of $S/M$ is the exact sequence
 \[\mathbb{T}_M:0\rightarrow F_q\xrightarrow{f_q}F_{q-1}\rightarrow\cdots\rightarrow F_1\xrightarrow{f_1}F_0\xrightarrow{f_0} 
 S/M\rightarrow0.\]
 \end{construction}
Following [Me], we define the \textbf{multidegree} of a Taylor symbol $[m_{i_1},\ldots,m_{i_s}]$, denoted $\mdeg[m_{i_1},\ldots,m_{i_s}]$, as follows:
  $\mdeg[m_{i_1},\ldots,m_{i_s}]=\lcm(m_{i_1},\ldots,m_{i_s})$.

\begin{definition}

 Let $M$ be a monomial ideal, and let
 \[\mathbb{F}:\cdots\rightarrow F_i\xrightarrow{f_i}F_{i-1}\rightarrow\cdots\rightarrow F_1\xrightarrow{f_1}F_0\xrightarrow{f_0} S/M\rightarrow 0\]
be a free resolution of $S/M$. 
We say that a basis element $[\sigma]$ of $\mathbb{F}$ has \textbf{homological degree} $i$, denoted $\hdeg[\sigma]=i$, if 
$[\sigma] \in F_i$. $\mathbb{F}$ is said to be a \textbf{minimal resolution} if for every $i$, the differential matrix $\left(f_i\right)$ of $\mathbb{F}$
has no invertible entries.
\end{definition}

\textit{Note}:
From now on, every time that we make reference to a free resolution $\mathbb{F}$ of $S/M$ we will assume that $\mathbb{F}$ is obtained from $\mathbb{T}_M$ by means of consecutive cancellations. To help us remember this convention, the basis elements of a free resolution will always be called Taylor symbols.

\begin{definition}
Let $M$ be a monomial ideal, and let
 \[\mathbb{F}:\cdots\rightarrow F_i\xrightarrow{f_i}F_{i-1}\rightarrow\cdots\rightarrow F_1\xrightarrow{f_1}F_0\xrightarrow{f_0} S/M\rightarrow 0\]
be a minimal free resolution of $S/M$.
\begin{itemize}
 \item For every $i$, the $i^{th}$ \textbf{Betti number} $\betti_i\left(S/M\right)$ of $S/M$ is $\betti_i\left(S/M\right)=\rank(F_i)$.
\item For every $i,j\geq 0$, the \textbf{graded Betti number} $\betti_{ij}\left(S/M\right)$ of $S/M$, in homological degree $i$ and internal degree $j$,
is \[\betti_{ij}\left(S/M\right)=\#\{\text{Taylor symbols }[\sigma]\text{ of }F_i:\deg[\sigma]=j\}.\]
\item For every $i\geq 0$, and every monomial $l$, the \textbf{multigraded Betti number} $\betti_{i,l}\left(S/M\right)$ of $S/M$, in homological degree $i$ and multidegree $l$,
is \[\betti_{i,l}\left(S/M\right)=\#\{\text{Taylor symbols }[\sigma]\text{ of }F_i:\mdeg[\sigma]=l\}.\]
\item The \textbf{projective dimension} $\pd\left(S/M\right)$ of $S/M$ is \[\pd\left(S/M\right)=\max\{i:\betti_i\left(S/M\right)\neq 0\}.\]
\end{itemize}
\end{definition}

Since the idea of polarization will play an important role in section 4, it is convenient to introduce some common notation. If $m=x_1^{\alpha_1}\cdots x_n^{\alpha_n}$ is a monomial in $S$, its polarization will be denoted $m'=x_{1,1}\cdots x_{1,\alpha_1}\cdots x_{n,1}\cdots x_{n,\alpha_n}$. If $G$ is a set of monomials, then the set of all polarizations $m'$ of monomials $m$ of $G$ will be denoted $G_{\pol}$. The polarization of a monomial ideal $M$ of $S$ will be denoted $M_{\pol}$. Finally, $M_{\pol}$ will be regarded as an ideal in $S_{\pol}$, the polynomial ring over $k$ whose variables are the ones that appear in the factorizations of the minimal generators of $M_{\pol}$.

We close this section with the concept of dominance [Al1, Al2] which is at the heart of this work.

\begin{definition}
Let $L$ be a set of monomials, and let $M$ be a monomial ideal with minimal generating $G$.
\begin{itemize}
\item An element $m\in L$ is a \textbf{dominant monomial} (in $L$) if there is a variable $x$, such that for all $m'\in L\setminus\{m\}$, the exponent with 
which $x$ appears in the factorization of $m$ is larger than the exponent with which $x$ appears in the factorization of $m'$. In this case, we say that $m$ is dominant in $x$, and $x$ is a \textbf{dominant variable} for $m$.
\item $L$ is called a \textbf{dominant set} if each of its monomials is dominant. 
\item $M$ is called a \textbf{dominant ideal} if $G$ is a dominant set. 
\item If $G'$ is a dominant set contained in $G$, we will say that $G'$ is a \textbf{dominant subset} of $G$.  (This does not mean that the elements of $G'$ are dominant in $G$, as the concept of dominant monomial always depends on a reference set.)

\end{itemize}
\end{definition}

\begin{example}\label{example 1}
  Let $M$ be minimally generated by $G=\{a^2b,ab^3c,bc^2,a^2c^2\}$, and let $G'=\{a^2b,ab^3c,bc^2\}$. Note that $ab^3c$ is the only dominant monomial in $G$, being $b$ a dominant variable for $ab^3c$. It is easy to check that
  $G'$ is a dominant set and, given that $G'\subseteq G$, $G'$ is a dominant subset of $G$. (Incidentally, notice that two of the dominant monomials in $G'$ are not dominant in $G$.) Finally, the ideal $M'$, minimally generated by $G'$, is a dominant ideal, for $G'$ is a dominant set. 
  \end{example}

\section{Order of Dominance}

\begin{definition}
Let $l=x_1^{\alpha_1}\cdots x_n^{\alpha_n}$, and $l'=x_1^{\beta_1}\cdots x_n^{\beta_n}$ be two monomials. We say that  $l$ \textbf{strongly divides} $l'$, 
 if $\alpha_i<\beta_i$, whenever $\alpha_i\neq 0$. 
\end{definition}

Let $a$, $b$ be elements of $S$, of the form $a=\alpha x_1^{\alpha_1}\cdots x_n^{\alpha_n}$, $b=\beta x_1^{\beta_1}\cdots x_n^{\beta_n}$, where $\alpha$, $\beta \in k$, and 
$\alpha_1,\ldots,\alpha_n,\beta_1,\ldots,\beta_n \geq 0$. We will say that $a$ and $b$ \textbf{have the same scalars} if $\alpha=\beta$.

\begin{lemma}\label{lemma 1}
Let $M$ be minimally generated by $G$, and let $\mathbb{F}$ be a minimal resolution of $S/M$. Let $\lcm(G)= x_1^{\alpha_1}\cdots x_n^{\alpha_n}$, and let 
$1\leq q\leq n$. Suppose that the following conditions hold:
\begin{enumerate}[(i)]
\item $G$ contains a dominant set $\{d_1,\ldots,d_q\}$ such that, for all $i=1,\ldots,q$, $d_i$ is dominant in $x_i$, and $x_i$ appears with exponent $\alpha_i$ in the factorization of $d_i$.
\item Every monomial in $G$ is divisible by at least one of $x_1^{\alpha_1},\ldots, x_q^{\alpha_q}$.
\end{enumerate}
Then there is a Taylor symbol $[\sigma]$ in the basis of $\mathbb{F}$ such that $\hdeg[\sigma]=q$, and $\mdeg[\sigma]=x_1^{\alpha_1}\cdots x_q^{\alpha_q} x_{q+1}^{\beta_{q+1}} \cdots x_n^{\beta_n}$, where 
$\beta_{q+1} \leq \alpha_{q+1}, \ldots, \beta_n\leq \alpha_n$.
\end{lemma}

\begin{proof}
Let $G'=G\cup\{x_{q+1}^{\alpha_{q+1}+1},\ldots, x_n^{\alpha_n+1}\}$, and let $M'$ be the ideal (minimally) generated by $G'$. Notice that the set $L=\{d_1,\ldots,d_q, x_{q+1}^{\alpha_{q+1}+1},\ldots,x_n^{\alpha_n+1}\}$ is dominant, of cardinality $n$, and such that each $d_i$ is dominant in $x_i$, and each $x_{q+i}^{\alpha_{q+i}+1}$ is dominant in $x_{q+i}$. Also,
\[\lcm(L)=\lcm(G')=x_1^{\alpha_1} \cdots x_q^{\alpha_q} x_{q+1}^{\alpha_{q+1}+1} \cdots x_n^{\alpha_n +1}.\]
By hypothesis, each element of $G$ is divisible by at least one of $x_1^{\alpha_1},\ldots,x_q^{\alpha_q}$. It follows that no monomial of $G'$ strongly divides $\lcm(L)$. \\
Let $A=\{[\sigma]\in \mathbb{T}_M: x_1^{\alpha_1} \cdots x_q^{\alpha_q} \mid \mdeg[\sigma]\}$, and $B=\{[\sigma] \in \mathbb{T}_{M'} :\mdeg[\sigma]=\lcm(L)\}$. Let 

\begin{align*}
f:A & \rightarrow  B\\
[l_1 ,\ldots, l_k] & \rightarrow  [l_1,\ldots,l_k,x_{q+1}^{\alpha_{q+1}+1},\ldots,x_n^{\alpha_n +1}].
\end{align*}

Notice that $f$ is bijective. Moreover, if we define $A_i=\{[\sigma]\in A : \hdeg[\sigma]=i\}$, and $B_i=\{[\sigma]\in B : \hdeg[\sigma]=i+(n-q)\}$, then $f$ defines a bijection between $A_i$ and $B_i$.\\
Suppose that $\mathbb{F}_M$ is a free resolution of $S/M$, obtained from $\mathbb{T}_M$ by doing the consecutive cancellations
\[
\begin{array}{cclclcc}
0&\rightarrow& S[\sigma_1]& \rightarrow & S[\tau_1] &\rightarrow &0,\\
  & \vdots\\
  0&\rightarrow& S[\sigma_r]& \rightarrow & S[\tau_r] &\rightarrow &0,
\end{array}\]
where $[\sigma_1],\ldots,[\sigma_r],[\tau_1],\ldots,[\tau_r]\in A$. \\
We will show that it is possible to obtain a free resolution $\mathbb{F}_{M'}$ (of $S/M'$) from $\mathbb{T}_{M'}$, by doing the cancellations
\[
\begin{array}{cclclcc}
0&\rightarrow& S f[\sigma_1]& \rightarrow & Sf[\tau_1] &\rightarrow &0,\\
  & \vdots\\
  0&\rightarrow& Sf[\sigma_r]& \rightarrow & Sf[\tau_r] &\rightarrow &0,
\end{array}\]
with the property that if $a_{\sigma\tau}^{(r)}$ is an entry of $\mathbb{F}_M$, determined by $[\sigma],[\tau]\in A$, then $f[\sigma],f[\tau]$ are in the basis of $\mathbb{F}_{M'}$, and the entry $b_{\tau\sigma}^{(r)}$, determined by them, has the same scalar as $a_{\tau\sigma}^{(r)}$. The proof is by induction on $r$.\\
If $r=0$, then $\mathbb{F}_M=\mathbb{T}_M$ and $\mathbb{F}_{M'}=\mathbb{T}_{M'}$. Thus, if $[\sigma],[\tau]\in A$, then $f[\sigma],f[\tau]\in \mathbb{F}_{M'}$. In particular, if $[\tau]$ is a facet of $[\sigma]$, $f[\tau]$ is a facet of 
$f[\sigma]$, and these Taylor symbols are of the form\\
$
[\sigma] =   [l_1,\ldots,l_k] , \quad \quad \quad \quad f[\sigma]= [l_1,\ldots,l_k,x_{q+1}^{\alpha_{q+1}+1},\ldots,x_n^{\alpha_n +1}]$\\
$[\tau] = [l_1,\ldots,\widehat{l_i},\ldots,l_k],  \quad f[\tau]  = [l_1,\ldots,\widehat{l_i},\ldots,l_k,x_{q+1}^{\alpha_{q+1}+1},\ldots,x_n^{\alpha_n +1}]. $\\
Therefore,
\[a_{\tau\sigma}^{(0)}=(-1)^{i+1}\dfrac{\mdeg[\sigma]}{\mdeg[\tau]}, \quad b_{\tau\sigma}^{(0)}=(-1)^{i+1}\dfrac{\mdeg f[\sigma]}{\mdeg f[\tau]}\]
On the other hand, if $[\tau] $ is not a facet of $[\sigma]$, $f[\tau]$ is not a facet of $f[\sigma]$, and $a_{\sigma\tau}^{(0)}=0=b_{\tau\sigma}^{(0)}$. In either case, $a_{\tau\sigma}^{(0)}$ and $b_{\tau\sigma}^{(0)}$ have the same scalar (either $(-1)^{i+1}$ or $0$).\\
Let us assume that our claim holds for $r-1$. Now, let us prove it for $r$.\\
Let $\mathbb{G}_M$ (respectively, $\mathbb{G}_{M'}$) be the resolution obtained from $\mathbb{T}_M$ (respectively, $\mathbb{T}_{M'}$) by doing the cancellations
\[0\rightarrow S[\sigma_i]\rightarrow S[\tau_i]\rightarrow 0, \quad i=1,\ldots,r-1\]
(respectively, $0\rightarrow Sf[\sigma_i] \rightarrow Sf[\tau_i]\rightarrow 0, \quad i=1,\ldots,r-1$).
By induction hypothesis, if $a_{\tau\sigma}^{(r-1)}$ is an entry of $\mathbb{G}_M$, determined by basis elements $[\sigma],[\tau]\in A$, then $f[\sigma],f[\tau]$ are in the basis of $\mathbb{G}_{M'}$, and the entry 
$b_{\tau\sigma}^{(r-1)}$ determined by them, has the same scalar as $a_{\tau\sigma}^{(r-1)}$. \\
Since neither $[\sigma_r]$ nor $[\tau_r]$ is in $\{[\sigma_1],\ldots,[\sigma_{r-1}],[\tau_1],\ldots,[\tau_{r-1}]\}$, it follows that neither $f[\sigma_r]$ nor $f[\tau_r]$ is in $\{f[\sigma_1],\ldots,f[\sigma_{r-1}],f[\tau_1],\ldots,f[\tau_{r-1}]\}$.
This means that $f[\sigma_r],f[\tau_r]$ are in the basis of $\mathbb{G}_{M'}$. Let us define $\mathbb{F}_M$ from $\mathbb{G}_M$ by doing the consecutive cancellation
\[0\rightarrow S[\sigma_r]\rightarrow S[\tau_r]\rightarrow 0.\]
Since $a_{\tau_r\sigma_r}^{(r-1)}$ is invertible, $a_{\tau_r\sigma_r}^{(r-1)}\neq 0$ and, given that $a_{\tau_r\sigma_r}^{(r-1)}$, $b_{\tau_r\sigma_r}^{(r-1)}$ have the same scalar, $b_{\tau_r\sigma_r}^{(r-1)}\neq 0$. In addition, since $\mdeg f[\sigma_r]=\mdeg f[\tau_r]$, we must have that $b_{\tau_r\sigma_r}^{(r-1)}\in k\setminus\{0\}$. Hence, we can define $\mathbb{F}_{M'}$ from $\mathbb{G}_{M'}$ by doing the consecutive cancellation
\[0\rightarrow S f[\sigma_r]\rightarrow S f[\tau_r]\rightarrow 0.\]
Moreover, if $a_{\tau\sigma}^{(r)}$ is an entry of $\mathbb{F}_M$, determined by two Taylor symbols $[\sigma],[\tau]\in A$, then neither $[\sigma]$ nor $[\tau]$ is in $\{[\sigma_1],\ldots,[\sigma_r],[\tau_1],\ldots,[\tau_r]\}$, and thus, neither $f[\sigma]$ nor $f[\tau]$ is in $\{f[\sigma_1],\ldots,f[\sigma_r],f[\tau_1],\ldots,f[\tau_r]\}$. This means that $f[\sigma]$, $f[\tau]$ are in the basis of $\mathbb{F}_{M'}$. Finally, we need to show that the entry 
$b_{\tau\sigma}^{(r)}$, determined by $f[\sigma],f[\tau]$ has the same scalar as $a_{\tau\sigma}^{(r)}$. By [Al1],
\[a_{\tau\sigma}^{(r)}=a_{\tau\sigma}^{(r-1)}-\dfrac{a_{\tau\sigma_r}^{(r-1)}a_{\tau_r\sigma}^{(r-1)}}{a_{\tau_r\sigma_r}^{(r-1)}}, \quad  \text{and}\quad b_{\tau\sigma}^{(r)}=b_{\tau\sigma}^{(r-1)}-\dfrac{b_{\tau\sigma_r}^{(r-1)}b_{\tau_r\sigma}^{(r-1)}}{b_{\tau_r\sigma_r}^{(r-1)}}.\]

By induction hypothesis, the pairs $a_{\tau\sigma}^{(r-1)}$ and $b_{\tau\sigma}^{(r-1)}$, $a_{\tau\sigma_r}^{(r-1)}$ and $b_{\tau\sigma_r}^{(r-1)}$, $a_{\tau_r\sigma}^{(r-1)}$ and $b_{\tau_r\sigma}^{(r-1)}$, 
$a_{\tau_r\sigma_r}^{(r-1)}$ and $b_{\tau_r\sigma_r}^{(r-1)}$, have the same respective scalars. Therefore, $a_{\tau\sigma}^{(r)}$ and $b_{\tau\sigma}^{(r)}$ have the same scalar, and our claim is proven.\\
Let us take $r$ to be as large as possible. That is, let $\mathbb{F}_M$ be obtained from $\mathbb{T}_M$ by means of $r$ consecutive cancellations between pairs of elements of $A$, with the property that if 
$a_{\tau\sigma}$ is an entry of $\mathbb{F}_M$, determined by Taylor symbols $[\sigma],[\tau]$ of $A$, then $a_{\tau\sigma}$ is noninvertible. It follows from [Al2, Theorem 3.3] that if $\mathbb{H}_M$ is a minimal resolution of 
$S/M$ obtained from $\mathbb{F}_M$ by means of consecutive cancellations, those consecutive cancellations do not involve elements of $A$. Suppose, by means of contradiction, that no Taylor symbol of $\mathbb{H}_M$ in homological degree $q$ belongs to $A$. This implies that no Taylor symbol of $\mathbb{F}_M$ in homological degree $q$ belongs to $A$. Since $f\restriction{A_q}: A_q\rightarrow B_q$ is a bijection that sends elements in homological degree $q$ to elements in homological degree $q+(n-q)=n$, it follows that no Taylor symbol of $\mathbb{F}_{M'}$ in homological degree $n$ belongs to $B$. Thus, if $\mathbb{H}_{M'}$ is a minimal resolution of 
$S/M'$, obtained from $\mathbb{F}_{M'}$ by means of consecutive cancellations, no Taylor symbol of $\mathbb{H}_{M'}$ in homological degree $n$ belongs to $B$. But given that the minimal generating set $G'$ of $M'$ contains the dominant set $L$, of cardinality $n$, such that no element of $G'$ strongly divides $\lcm(L)$, it follows from [Al2, Theorem 5.2] that $\mathbb{H}_{M'}$ must contain a basis element $[\theta]$, such that $\hdeg[\theta]=n$, and $\mdeg[\theta]=\lcm(L)$, a contradiction.\\
We conclude that the minimal resolution of $S/M$ contains a Taylor symbol $[\sigma]$, such that $\hdeg[\sigma]=q$, and $[\sigma]\in A$.
\end{proof}

\begin{lemma}\label{lemma 2}
Let $M$ be minimally generated by $G$, and let $\mathbb{F}$ be a minimal resolution of $S/M$. Let $\lcm(G)=x_1^{\alpha_1} \cdots x_n^{\alpha_n}$, and let $1\leq i_1< \cdots < i_q\leq n$, where $1\leq q\leq n$. Suppose that the following conditions hold:
\begin{enumerate}[(i)]
\item $G$ contains a dominant set $\{d_1,\ldots,d_q\}$ such that, for all $j=1,\ldots, q$, $d_j$ is dominant in $x_{i_j}$, and $x_{i_j}$ appears with exponent $\alpha_{i_j}$ in the factorization of $d_j$.
\item Every monomial in $G$ is divisible by at least one of $x_{i_1}^{\alpha_{i_1}},\ldots, x_{i_q}^{\alpha_{i_q}}$.
\end{enumerate}
Then there is a Taylor symbol $[\sigma]$ in the basis of $\mathbb{F}$ such that $\hdeg[\sigma]=q$, and $\mdeg[\sigma]=x_1^{\lambda_1} \cdots x_n^{\lambda_n}$, where $\lambda_j=\alpha_j$ if $j\in \{i_1,\ldots,i_q\}$, and $\lambda_j \leq \alpha_j$ if $j \notin \{i_1,\ldots,i_q\}$.
\end{lemma}

\begin{proof}
Let $f$ be a permutation of $\{1,\ldots,n\}$, such that $f(j)=i_j$, for all $j=1,\ldots,q$. For all $j=1,\ldots,n$, we define $y_j=x_{f(j)}$; and $\delta_j=\alpha_{f(j)}$. Then $\lcm(G)=y_1^{\delta_1}\cdots y_n^{\delta_n}$; each 
$d_j$ is dominant in $y_j$, and $y_j$ appears with exponent $\delta_j$ in the factorization of $d_j$. Moreover, each monomial in $G$ is divisible by at least one of $y_1^{\delta_1},\ldots,y_q^{\delta_q}$. By Lemma \ref{lemma 1}, 
$\mathbb{F}$ has a basis element $[\sigma]$, such that $\hdeg[\sigma]=q$ and $\mdeg[\sigma]=y_1^{\delta_1} \cdots y_q^{\delta_q} y_{q+1}^{\epsilon_{q+1}} \cdots y_n^{\epsilon_n}$, where 
$\epsilon_{q+1}\leq \delta_{q+1}, \ldots, \epsilon_n\leq \delta_n$. With our original terminology, this means that $\mathbb{F}$ has a Taylor symbol $[\sigma]$, such that $\hdeg[\sigma]=q$, and 
$\mdeg[\sigma]= x_1^{\lambda_1} \cdots x_n^{\lambda_n}$, where $\lambda_j=\alpha_j$ if $j\in \{i_1,\ldots,i_q\}$, and $\lambda_j\leq \alpha_j$ if $j\notin \{i_1,\ldots,i_q\}$.
\end{proof}

\begin{definition}\label{comb.def.}
Let $M$ be minimally generated by $G$. Let $\mathscr{I}$ be the class of all sequences $1\leq i_1< \cdots <i_q \leq n$, with $q\geq 1$, such that the following conditions hold:
\begin{enumerate}[(i)]
\item $G$ contains a dominant set $D=\{d_1,\ldots,d_q\}$ such that each $d_j$ is dominant in $x_{i_j}$.
\item If $\lcm(D)=x_{1}^{\alpha_{1}} \cdots x_{n}^{\alpha_{n}}$, and $m$ is an element of $G$ dividing $\lcm(D)$, then $m$ is divisible by at least one of $x_{i_1}^{\alpha_{i_1}},\ldots,x_{i_q}^{\alpha_{i_q}}$.
\end{enumerate}
We define  the \textbf{order of dominance} of $S/M$, denoted $\odom (S/M)$, as the maximum of the cardinalities of the sequences of $\mathscr{I}$. 
\end{definition}

In the next section we give a more intuitive definition of order of dominance. (The definition in the next section is much more similar to the one given in the introduction.)

\begin{theorem}\label{Theorem 3.4}
Let $\odom (S/M)=q$. Then, $\pd(S/M)\geq q$, and $\betti_r(S/M)\geq {q\choose r}$, for all $r$.
\end{theorem}

\begin{proof}
Let $G$ be the minimal generating set of $M$. By definition, there is a sequence $1\leq i_1< \cdots <i_q \leq n$, with $q\geq 1$, such that:
\begin{enumerate}[(i)]
\item $G$ contains a dominant set $\{d_1,\ldots,d_q\}$ such that, for all $j=1,\ldots,q$, $d_j$ is dominant in $x_{i_j}$, and $x_{i_j}$ appears with exponent $\alpha_{i_j}\geq 1$ in the factorization of $d_j$.
\item If $m \in G$, and $m\mid \lcm(d_1,\ldots,d_q)$, then $m$ is divisible by at least one of $x_{i_1}^{\alpha_{i_1}}, \ldots, x_{i_q}^{\alpha_{i_q}}$.
\end{enumerate} 
Let $i_{j_1},\ldots, i_{j_r}$ be a subsequence of $i_1,\ldots i_q$. Then $\lcm(d_{j_1},\ldots,d_{j_r})=x_1^{\beta_1}\cdots x_n^{\beta_n}$, where $\beta_k=\alpha_k$ if $k\in\{i_{j_1},\ldots,i_{j_r}\}$, and $\beta_k<\alpha_k$ if 
$k\in \{i_1,\ldots,i_q\}\setminus \{i_{j_1},\ldots,i_{j_r}\}$. Let $G'=\{m\in G: m\mid \lcm(d_{j_1},\ldots,d_{j_r})\}$. Notice that $\{d_{j_1},\ldots,d_{j_r}\}$ is a dominant set of cardinality $r$, such that each $d_{j_k}$ is dominant in $x_{i_{j_k}}$, and each $x_{i_{j_k}}$ appears with exponent $\alpha_{i_{j_k}}$ in the factorization of $d_{j_k}$. In addition, every $m\in G'$ is divisible by at least one of $x_{i_{j_1}}^{\alpha_{i_{j_1}}},\ldots,x_{i_{j_r}}^{\alpha_{i_{j_r}}}$ because, otherwise, $m$ could not be divisible by any of $x_{i_1}^{\alpha_{i_1}},\ldots, x_{i_q}^{\alpha_{i_q}}$, which contradicts property (ii).\\
Let $M'$ be the ideal (minimally) generated by $G'$, and let $\mathbb{F}_{M'}$ be a minimal resolution of $S/M'$. By Lemma \ref{lemma 2}, $\mathbb{F}_{M'}$ contains a Taylor symbol $[\sigma]$, such that $\hdeg[\sigma]=r$, and $\mdeg[\sigma]=x_1^{\lambda_1}\cdots x_n^{\lambda_n}$, where $\lambda_k=\beta_k$ if $k\in\{i_{j_1},\ldots,i_{j_r}\}$, and $\lambda_k\leq \beta_k$ if $k\notin \{i_{j_1},\ldots,i_{j_r}\}$. In particular $\lambda_k=\alpha_k$ if 
$k\in \{i_{j_1},\ldots, i_{j_r}\}$, and $\lambda_k< \alpha_k$ if $k\in \{i_1,\ldots,i_q\}\setminus \{i_{j_1},\ldots, i_{j_r}\}$. Since $\mathbb{F}_{M'}$ is a subresolution of the minimal resolution $\mathbb{F}_M$ of $S/M$ [GHP], we conclude that $\mathbb{F}_M$ contains a Taylor symbol $[\sigma]$, such that $\hdeg[\sigma]=r$, and $\mdeg[\sigma]=x_1^{\lambda_1} \cdots x_n^{\lambda_n}$, where $\lambda_k=\alpha_k$ if $k\in \{i_{j_1},\ldots,i_{j_r}\}$, and 
$\lambda_k< \alpha_k$, if $k\in \{i_1,\ldots,i_q\}\setminus \{i_{j_1},\ldots,i_{j_r}\}$.\\
By defining $[\sigma_{i_{j_1},\ldots,i_{j_r}}]=[\sigma]$, we can establish a correspondence between the set $A_r$ of all subsequences $i_{j_1},\ldots,i_{j_r}$ of $i_1,\ldots,i_q$, and the set $B_r$ of Taylor symbols of 
$\mathbb{F}_M$ in homological degree $r$. Let 
\begin{align*}
f:A_r & \rightarrow  B_r\\
\{i_{j_1} ,\ldots, i_{j_r}\} & \rightarrow  [\sigma_{i_{j_1},\ldots,i_{j_r}}]
\end{align*}
be the function that determines this correspondence. If $\{i_{j_1},\ldots,i_{j_r}\}$ and $\{i_{t_1},\ldots,i_{t_r}\}$ are different sequences of $A_r$, then there must be an index that belongs to one sequence but not to the other; say 
$\lambda_{s_k}\in \{i_{s_1},\ldots,i_{s_r}\}\setminus \{i_{t_1},\ldots,i_{t_r}\}$. Hence, $x_{i_{s_k}}^{\alpha_{i_{s_k}}}\mid \mdeg[\sigma_{i_{s_1},\ldots,i_{s_r}}]$, and $x_{i_{s_k}}^{\alpha_{i_{s_k}}}\nmid \mdeg[\sigma_{i_{t_1},\ldots,i_{t_r}}]$. This implies that $\mdeg[\sigma_{i_{s_1},\ldots,i_{s_r}}]\neq \mdeg[\sigma_{i_{t_1},\ldots,i_{t_r}}]$. In particular, $[\sigma_{i_{s_1},\ldots,i_{s_r}}]\neq [\sigma_{i_{t_1},\ldots,i_{t_r}}]$, and thus, $f$ is one-to-one. Finally,
\[\betti_r(S/M)=\# B_r \geq \# A_r={q\choose r}\]
and given that $r$ is arbitrary, the theorem holds.
\end{proof}

The second part of Theorem \ref{Theorem 3.4} is weaker than a result of Brum and R$\ddot{o}$mer [BR], asserting that $\betti_r(S/M)\geq {\pd(S/M)\choose r}$. Combining this last inequality with the first part of Theorem \ref{Theorem 3.4}, we obtain that $\betti_r(S/M)\geq {\pd(S/M)\choose r}\geq {q\choose r}$. However, the weaker inequality stated in Theorem \ref{Theorem 3.4} is useful in those instances where the order of dominance is known, but the projective dimension is unknown (cf. Example \ref{odom by nets}).

In Theorem \ref{nets 5}, we will give a list of five equivalent statements that characterize the class of monomial ideals $M$ for which $\pd(S/M)=n$. The next theorem is key to such characterization. 

\begin{theorem}\label{Theorem 3.5}
Let $M$ be minimally generated by $G$. Suppose that $G$ contains a dominant set $L$, of cardinality $n$, such that no monomial of $G$ strongly divides $\lcm(L)$. Then $\betti_r(S/M)\geq {n\choose r}$, for all $r$.
\end{theorem}
\begin{proof}
Let $L=\{d_1,\ldots,d_n\}$, where each $d_j$ is dominant in $x_j$. Let $\lcm(L)= x_1^{\alpha_1}\cdots x_n^{\alpha_n}$. Then, each $x_i$ appears with exponent $\alpha_i$ in the factorization of $d_i$. Since no minimal generator strongly divides $\lcm(L)$, every $m \in G$ must be divisible by at least one of  $x_1^{\alpha_1},\ldots, x_n^{\alpha_n}$. Thus, $\odom (S/M)=n$, and the result follows from Theorem \ref{Theorem 3.4}.
\end{proof}

Now we use the idea of order of dominance to compute projective dimensions. Recall that the Scarf complex [BPS] is a subcomplex of the Taylor resolution and, in general, it does not give a resolution of $S/M$, but when it does, it gives a minimal resolution of $S/M$. As shown in [BPS], the basis of the Scarf complex consists of each Taylor symbol whose multidegree is different from the multidegrees of the other Taylor symbols.

 Monomial ideals minimally resolved by the Scarf complex are called Scarf ideals (generic [Mi], strongly generic [BPS], dominant [Al1], and semidominant [Al1] ideals are examples of Scarf ideals). In the next theorem, we express the projective dimension of Scarf ideals in terms of their order of dominance, and in the next section, we use the order of dominance to characterize the Scarf ideals that are Cohen-Macaulay (Corollary \ref{C-M}).

\begin{theorem}\label{ScarfPD}
Let $M$ be a Scarf ideal. Then $\pd(S/M)=\odom(S/M)$.
\end{theorem}
\begin{proof}
Let $M=(m_1, \ldots, m_q)$. The basis of the Scarf complex of $S/M$ is

\[\mathscr{F} = \{[\sigma] \in \mathbb{T}_M: \text{ if } [\tau]\in \mathbb{T}_M \text{ and }\mdeg[\sigma] = \mdeg[\tau]\text{, then }[\sigma] = [\tau]\}.\]
Therefore, $\pd(S/M) = \Max\{\hdeg[\sigma]:[\sigma]\in\mathscr{F}\}$. Denote $t=\pd(S/M)$. Then there is an element $[\sigma]\in\mathscr{F}$ of the form $[\sigma] = [m_{i_1},\ldots,m_{i_t}]$. We will show that the set $\{m_{i_1},\ldots,m_{i_t}\}$ is dominant. Suppose not. Then there is $1\leq r\leq t$ such that $m_{i_r}\mid \lcm(m_{i_1},\ldots,\widehat{m_{i_r}},\ldots,m_{i_t})$. It follows that $\mdeg[m_{i_1},\ldots,\widehat{m_{i_r}},\ldots,m_{i_t}] = \mdeg[m_{i_1},\ldots,m_{i_t}]$, a contradiction. Now, suppose that there is a minimal generator $m\in \{m_1,\ldots,m_q\} \setminus \{m_{i_1},\ldots,m_{i_t}\}$, such that $m\mid \lcm(m_{i_1},\ldots,m_{i_t})$. Then $\mdeg[m_{i_1},\ldots,m_{i_t}]=\mdeg[m_{i_1},\ldots,m_{i_t},m]$, another contradiction.

Summarizing, $\{m_{i_1},\ldots,m_{i_t}\}$ is a dominant set such that, if a minimal generator $m$ divides $\lcm(m_{i_1},\ldots,m_{i_t})$, then $m\in \{m_{i_1},\ldots,m_{i_t}\}$. It follows from Definition \ref{comb.def.} that $\odom(S/M)\geq t$. Finally, $t\leq\odom(S/M)\leq \pd(S/M) = t$.

\end{proof}

We close this section with a characterization of when the Taylor resolution is minimal.

\begin{theorem}\label{Taylor}
Let $M=(m_1,\ldots,m_q)$ be a monomial ideal. Then $\mathbb{T}_M$ is a minimal resolution of $S/M$ if and only if $\odom (S/M)=q$. 
\end{theorem}
\begin{proof}
By [Al1, Theorem 4.4], $\mathbb{T}_M$ is minimal if and only if $\{m_1,\ldots,m_q \}$ is a dominant set, which is equivalent to saying that $\odom (S/M)=q$.
\end{proof}

\section{Nets}

\begin{definition}
Let $M$ be minimally generated by $G$. Let $X=\{x_{i_1},\ldots,x_{i_q}\}$ be a set of variables of $S$. We will say that $X$ is a \textbf{net} of $M$ if every monomial of $G$ is divisible by at least one element of $X$. We will say that $X$ is a \textbf{minimal net} of $M$, if $X$ itself is a net of $M$, but no proper subset of $X$ is a net of $M$.
\end{definition}

\begin{example}
Let $M=(a^2e, b^3f, ce^2, d^2f^3)$. Let $X_1=\{a,b,c,d\}$; $X_2=\{e,f\}$; $X_3=\{a,c,f\}$; $X_4=\{b,d,c\}$; $X_5=\{d,e,f\}$; $X_6=\{b,d,e,f\}$. Note that $X_1$, $X_2$, $X_3$, and $X_4$ are minimal nets of $M$ (and there are no other minimal nets of $M$), while $X_5$ and $X_6$ are nets (but not minimal nets) of $M$.
\end{example} 

\begin{theorem}\label{nets 1}
Let $M$ be minimally generated by $G$. Let $X=\{x_{i_1},\ldots,x_{i_q}\}$ be a minimal net of $M$.Then, $G$ contains a dominant set $D=\{d_1,\ldots,d_q\}$ with the following properties:
\begin{enumerate}[(i)]
\item For all $j=1,\ldots, q$, $d_j$ is dominant in $x_{i_j}$.
\item  If $\lcm(D)=x_1^{\alpha_1}\cdots x_n^{\alpha_n}$, and $m$ is an element of $G$ dividing $\lcm(D)$, then $m$ is divisible by one of $x_{i_1}^{\alpha_{i_1}},\ldots,x_{i_q}^{\alpha_{i_q}}$.
\end{enumerate}
\end{theorem}

\begin{proof}
By definition, every monomial of $G$ is divisible by some variable of $X$. In addition, if $1\leq j \leq q$, since $X$ is a minimal net, not every monomial of $G$ is divisible by a variable of 
$X\setminus \{x_{i_j}\}$. Therefore, there exists $m_j \in G$ such that $m_j$ is divisible by $x_{i_j}$, but $m_j$ is not divisible by any monomial of $X\setminus \{x_{i_j}\}$. Thus, there is a dominant set 
$\{m_1,\ldots,m_q\} \subseteq G$ such that, for all $j=1,\ldots, q$, $m_j$ is divisible by $x_{i_j}$, but is not divisible by any variable of $X\setminus\{x_{i_j}\}$.\\
Let $G_1=\{m\in G:x_{i_1} \mid m$ but $m$ is not divisible by any of $x_{i_2},\ldots, x_{i_q}\}$. (Notice that $G_1\neq \varnothing$, for $m_1 \in G_1$.) Let $\epsilon_1 \geq 1$ be the smallest exponent with which $x_{i_1}$ appears in the factorization of an element of $G_1$, and let $d_1\in G_1$ be such that $x_{i_1}$ appears with exponent $\epsilon_1$ in the factorization of $d_1$.
Let $G_2=\{m\in G: x_{i_2}\mid m$ but $m$ is not divisible by any of $x_{i_1}^{\epsilon_1},x_{i_3},\ldots,x_{i_q}\}$ ($G_2\neq \varnothing$, for $m_2\in G_2$.)\\
Let $\epsilon_2\geq 1$ be the smallest exponent with which $x_{i_2}$ appears in the factorization of an element of $G_2$, and let $d_2\in G_2$ be such that the exponent with which $x_{i_2}$ appears in the factorization of $d_2$ is $\epsilon_2$.\\
Suppose that $G_{k-1}$ and $d_{k-1}$ have been defined. Let $G_k=\{m\in G: x_{i_k}\mid m$, but $m$ is not divisible by any of $x_{i_1}^{\epsilon_1},\ldots, x_{i_{k-1}}^{\epsilon_{k-1}},x_{k+1}^{\epsilon_{k+1}},\ldots,x_{i_q}\}$. 
(Notice that $G_k\neq \varnothing$, for $m_k\in G_k$.) \\
Let $\epsilon_k$ be the smallest exponent with which $x_{i_k}$ appears in the factorization of an element of $G_k$. Let $d_k\in G_k$ be such that $x_{i_k}$ appears with exponent $\epsilon_k$ in the factorization of $d_k$.\\
By recurrence, we have constructed a set $D=\{d_1,\ldots,d_q\}$, with the following properties:
\begin{enumerate}[(i)]
\item $D$ is dominant; each $d_j$ is dominant in $x_{i_j}$; and $x_{i_j}$ appears with exponent $\epsilon_j$ in the factorization of $d_j$. In fact, by construction, $x_{i_j}$ appears with exponent $\epsilon_j$ in the factorization of $d_j$; $d_1,\ldots,d_{j-1}$ are not divisible by $x_{i_j}$; and $d_{j+1},\ldots,d_q$ are not divisible by $x_{i_j}^{\epsilon_j}$.
\item Every $m\in G$ dividing $\lcm(D)$, is divisible by one of $x_{i_1}^{\epsilon_1},\ldots,x_{i_q}^{\epsilon_q}$. In fact, if $m\in G$, and $m\mid \lcm(D)$, there is a variable of $X$ that divide $m$. Let $j$ be the largest number in $\{1,\ldots,q\}$ such that $x_{i_j}$ divides $m$. Then, either one of the $x_{i_1}^{\epsilon_1},\ldots,x_{i_{j-1}}^{\epsilon_{j-1}}$ divides $m$, or $m\in G_j$, in which case $x_{i_j}^{\epsilon_j}\mid m$. 
\end{enumerate}
\end{proof}

\begin{corollary} \label{nets 1'}
If $X$ is a minimal net of $M$, then $\odom(S/M)\geq \# X$. 
\end{corollary}

\begin{proof}
Let $X=\{x_{i_1},\ldots,x_{i_q}\}$. By Theorem \ref{nets 1}, the minimal generating set $G$ of $M$ contains a dominant set $D=\{d_1,\ldots,d_q\}$ with the following properties:
\begin{enumerate}[(i)]
\item For all $j=1,\ldots, q$, $d_j$ is dominant in $x_{i_j}$.
\item  If $\lcm(D)=x_1^{\alpha_1}\cdots x_n^{\alpha_n}$, and $m$ is an element of $G$ dividing $\lcm(D)$, then $m$ is divisible by one of $x_{i_1}^{\alpha_{i_1}},\ldots,x_{i_q}^{\alpha_{i_q}}$.
\end{enumerate}
By definition of order of dominance, $\odom(S/M)\geq \ q= \# X$.
\end{proof}

\begin{theorem}\label{nets 2}
If $X$ is a minimal net of $M_{\pol}$, then $\odom(S/M)\geq \# X$.
\end{theorem}

\begin{proof}
Let $S_{\pol}=k[x_{1,1},\ldots,x_{1,\beta_1},\ldots,x_{n,1},\ldots,x_{n,\beta_n}]$. Suppose that $X$ contains two variables of the form $x_{i,r}$, $x_{i,s}$, where $r>s$. Then every monomial of $G_{\pol}$ that is divisible by $x_{i,r}$ is also divisible by $x_{i,s}$ and thus, $X\setminus\{x_{i,r}\}$ is also a net of $M_{\pol}$, which contradicts the minimality of $X$. Hence, $X$ must be of the form $X=\{x_{i_1,r_1},\ldots,x_{i_q,r_q}\}$, where 
$1\leq i_1<\ldots<i_q\leq n$, and $r_1,\ldots,r_q\geq 1$. By Theorem \ref{nets 1}, there is a dominant subset $L_{\pol}=\{l'_1,\ldots,l'_q\}$ of $G_{\pol}$ such that, if $\lcm(L_{\pol})=x_{1,1}\ldots x_{1,\alpha_1}\ldots x_{n,1}\ldots x_{n,\alpha_n}$, the following properties hold:
\begin{enumerate}[(i)]
\item Each $l'_j$ is dominant in $x_{i_j,r_j}$, and $x_{i_j,r_j}$ appears with exponent $1$ in the factorization of $l'_j$.
\item If $m' \in G_{\pol}$ and $m' \mid \lcm(L_{\pol})$, then $m'$ is divisible by at least one of $x_{i_1,r_1},\ldots,x_{i_q,r_q}$.
\end{enumerate}
Let $L=\{l_1,\ldots,l_q\}$ be the subset of $G$ obtained from $L_{\pol}$ by depolarizing. The fact that $L_{\pol}$ is dominant, with $l'_j$ dominant in $x_{i_j,r_j}$, means that for all $k\neq j$, $x_{i_j,r_j}\nmid l'_k$. It follows that 
$x_{i_j}^{r_j}\nmid l_k$, for all $k\neq j$. Then $x_{i_j}^{\alpha_j}\nmid l_k$, for all $k\neq j$. Therefore, $L=\{l_1,\ldots,l_q\}$ is a dominant set, such that
\begin{enumerate}[(i)]
\item Each $l_j$ is dominant in $x_{i_j}$, and $x_{i_j}$ appears with exponent $\alpha_j$ $(\geq r_j)$ in the factorization of $l_j$.
\item If $m\in G$ and $m\mid \lcm(L)$, then $m$ is divisible by at least one of $x_{i_1}^{r_1},\ldots,x_{i_q}^{r_q}$.
\end{enumerate}
Let $G_1=\{m\in G: m\mid \lcm(L); x_{i_1}^{r_1}\mid m$; and $m$ is not divisible by any of $x_{i_2}^{r_2},\ldots,x_{i_q}^{r_q}\}$. ($G_1\neq \varnothing$, for $l_1\in G_1$.) Let $t_1$ be the largest integer such that every $m\in G_1$ is divisible by $x_{i_1}^{t_1}$. Let $d_1\in G_1$ be such that $x_{i_1}$ appears with exponent $t_1$ in the factorization of $d_1$.\\
Let $G_2=\{m\in G: m\mid \lcm(L); x_{i_2}^{r_2}\mid m$; and $m$ is not divisible by any of $x_{i_1}^{t_1},x_{i_3}^{r_3},\ldots,x_{i_q}^{r_q}\}$. ($G_2\neq \varnothing$, for $l_2\in G_2$.) Let $t_2$ be the largest integer such that every $m\in G_2$ is divisible by $x_{i_1}^{t_2}$. Let $d_2\in G_2$ be such that $x_{i_2}$ appears with exponent $t_2$ in the factorization of $d_2$.\\
Suppose that $G_{k-1}$, $t_{k-1}$, and $d_{k-1}$ have been defined.\\
Let $G_k=\{m\in G: m\mid \lcm(L); x_{i_k}^{r_k}\mid m$; and $m$ is not divisible by any of $x_{i_1}^{t_1},\ldots,x_{i_{k-1}}^{t_{k-1}},x_{i_{k+1}}^{r_{k+1}},$\\
$\ldots,x_{i_q}^{r_q}\}$. ($G_k\neq \varnothing$, for $l_k\in G_k$.) Let $t_k$ be the largest integer such that every $m\in G_k$ is divisible by $x_{i_k}^{t_k}$. Let $d_k\in G_k$ be such that $x_{i_k}$ appears with exponent $t_k$ in the factorization of $d_k$.\\
Thus, by recurrence, we can define a dominant subset $D=\{d_1,\ldots,d_q\}$ of $G$, such that 
\begin{enumerate}[(i)]
\item each $d_j$ is dominant in $x_{i_j}$, and $x_{i_j}$ appears with exponent $t_j$ in the factorization of $d_j$. We claim that $D$ also has the following property:
\item if $m\in G$ and $m\mid \lcm(D)$, then $m$ is divisible by at least one of $x_{i_1}^{t_1},\ldots, x_{i_q}^{t_q}$.
\end{enumerate}
In fact, if $m \in G$ and $m\mid \lcm(D)$, then $m\mid \lcm(L)$ (by construction, each $d_j\in D$ divides $\lcm(L)$). Then $m$ must be divisible by one of $x_{i_1}^{r_1},\ldots,x_{i_q}^{r_q}$. Let $k$ be the largest number among 
$1,\ldots,q$, such that $x_{i_k}^{r_k}\mid m$. If $m$ is not divisible by one of $x_{i_1}^{t_1},\ldots, x_{i_{k-1}}^{t_{k-1}}$, then $m\in G_k$. Thus, $x_{i_k}^{t_k}\mid m$, and (ii) holds. By Theorem \ref{Theorem 3.4}, $\odom(S/M)\geq q=\# X$.
\end{proof}

\begin{theorem}\label{nets 3}
If $\odom(S/M)=q$, then there is a minimal net $X$ of $M_{\pol}$ such that $\# X=q$.
\end{theorem}

\begin{proof}
By definition of order of dominance, for some sequence $1\leq i_1<\ldots<i_q\leq n$, there is a dominant subset $D=\{d_1,\ldots,d_q\}$ of $G$, with the following properties:
\begin{enumerate}[(i)]
\item each $d_j$ is dominant in $x_{i_j}$;
\item if $\lcm(D)=x_1^{\alpha_1}\ldots x_n^{\alpha_n}$, and $m$ is an element of $G$ dividing $\lcm(D)$, then $m$ is divisible by one of $x_{i_1}^{\alpha_{i_1}},\ldots,x_{i_q}^{\alpha_{i_q}}$.
\end{enumerate}
We claim that the set $G'=\{m\in G: m$ is not divisible by any of $x_{i_1}^{\alpha_{i_1}},\ldots,x_{i_q}^{\alpha_{i_q}}\}$ is empty. \\
Suppose, by means of contradiction, that $G'\neq\varnothing$. Let $G''=\{\lcm(D\cup\{m\}): m\in G'\}$. Let $\lcm(D\cup \{d\})$ be a minimal element of $G''$, in the sense that if $\lcm(D\cup \{m\}) \mid \lcm(D\cup \{d\})$, then $\lcm(D\cup \{m\})=\lcm(D\cup \{d\})$.\\
Note that $D\cup \{d\}$ is a dominant set. In fact, since $d$ is not divisible by any of $x_{i_1}^{\alpha_{i_1}},\ldots,x_{i_q}^{\alpha_{i_q}}$, by property (ii) we must have that $d\nmid \lcm(D)$. Hence, there exist 
$i\in \{1,\ldots,n\} \setminus \{i_1,\ldots,i_q\}$ such that $x_i$ appears with exponent $\delta> \alpha_i$ in the factorization of $d$.\\
Thus, $D\cup \{d\}$ is a dominant set such that 
\begin{enumerate}[(i)]
\item each $d_j$ is dominant in $x_{i_j}$, and $x_{i_j}$ appears with exponent $\alpha_{i_j}$ in the factorization of $d_j$. Also, $d$ is dominant in $x_i$, and $x_i$ appears with exponent $\delta$ in the factorization of $d$.
\item if $m\in G$ and $m\mid \lcm(D\cup \{d\})$, then $m$ is divisible by at least one of $x_{i_1}^{\alpha_{i_1}},\ldots,x_{i_q}^{\alpha_{i_q}},x_i^{\delta}$.  
\end{enumerate}
In fact, if such an $m$ is not divisible by any of $x_{i_1}^{\alpha_{i_1}},\ldots,x_{i_q}^{\alpha_{i_q}}$, then $m\in G'$, and $\lcm(D\cup\{m\})\in G''$. Since $\lcm(D\cup\{m\})\mid \lcm(D\cup\{d\})$, the minimality of $\lcm(D\cup\{d\})$ implies that $\lcm(D\cup\{m\})= \lcm(D \cup \{d\})$. Hence, $m$ must be divisible by $x_i^{\delta}$.\\
Thus, $\odom(S/M)\geq \#(D\cup\{d\})=q+1$, an absurd.\\
Therefore, $G'=\varnothing$, which means that every $m\in G$ is divisible by at least one of $x_{i_1}^{\alpha{i_1}},\ldots,x_{i_q}^{\alpha_{i_q}}$. After polarizing $D$, we obtain a dominant set $D_{\pol}=\{d'_1,\ldots,d'_q\}$, such that each $d'_j$ is dominant in $x_{i_j,\alpha_{i_j}}$. Also, every $m'\in G_{\pol}$ is divisible by at least one of $x_{i_1,\alpha_{i_1}},\ldots,x_{i_q,\alpha_{i_q}}$. Thus, $X=\{x_{i_1,\alpha_{i_1}},\ldots,x_{i_q,\alpha_{i_q}}\}$ is a minimal net of $M_{\pol}$.
\end{proof}

\begin{corollary} \label{nets 4}
Let $\Xi$ be the class of all minimal nets of $M_{\pol}$. Then 
\[\odom(S/M)=\max\{\#X: X\in \Xi\}.\]
\end{corollary}

\begin{proof}
By Theorem \ref{nets 2}, $\odom(S/M)\geq \max \{\# X: X\in \Xi\}$. By Theorem \ref{nets 3}, $\odom(S/M)\leq \max \{\#X: X\in \Xi\}$.
\end{proof}

Note that Corollary \ref{nets 4} gives an alternative definition of the concept of order of dominance; that is, $\odom(S/M)$ is the largest cardinality of a minimal net of $M_{\pol}$.

\begin{example}\label{odom by nets}
Let $M=(ae, be, ce, de, ab, cd)$. Note that the set $\{a,b,c,d\}$ is a minimal net of $M$. Since $\{a, b, c, d, e\}$ is not a minimal net of $M$, we must have that $\odom(S/M)=4$. More generally, suppose that $S=k[x_1,\ldots,x_n]$, and $M$ is a squarefree monomial ideal. If $x_1x_n,\ldots,x_{n-1}x_n$ are among the minimal generators of $M$, then $\{x_1,\ldots,x_{n-1} \}$ is a minimal net of $M$, and since $\{x_1,\ldots,x_{n} \}$ is not a minimal net of $M$, $\odom(S/M)=n-1$.
\end{example} 

\begin{corollary}\label{C-M} 
Let $M$ be a Scarf ideal. Then $M$ is Cohen-Macaulay if and only if all minimal nets of $M_{\pol}$ have the same cardinality.
\end{corollary}

\begin{proof}
By Theorem \ref{ScarfPD}, $M$ is Cohen-Macaulay if and only if $\codim(S/M)=\odom(S/M)$. Now the result follows by noting that $\codim(S/M)$ and $\odom(S/M)$ are the shortest and largest cardinalities of minimal nets of $M_{\pol}$, respectively.
\end{proof}

\section{Further results}

In this section, we state and prove some of the main properties of the order of dominance.

\begin{theorem}\label{nets 5}
Let $M$ be a monomial ideal of $S=k[x_1,\ldots,x_n]$, minimally generated by $G$. The following are equivalent:
\begin{enumerate}[(i)]
\item $\pd(S/M)=n$.
\item $\betti_i(S/M)\geq {n\choose i}$, for all $i$.
\item $G$ contains a dominant set $D$ of cardinality $n$, such that no monomial of $G$ strongly divides $\lcm(D)$.
\item $\odom(S/M)=n$.
\item There is a minimal net of $M_{\pol}$, of cardinality $n$.
\end{enumerate}
\end{theorem}

\begin{proof}
(i)$\Leftrightarrow$ (iii) [Al2, Corollary 5.3].\\
(iii)$\Rightarrow$ (iv) Let $\lcm(D)=x_1^{\alpha_1} \cdots x_n^{\alpha_n}$. Since $D$ is dominant, and $\# D=n$, $D$ can be expressed in the form $D=\{d_1,\ldots,d_n\}$, where each $d_j$ is dominant in $x_j$. Then, the exponent with which $x_j$ appears in the factorization of $d_j$ must be $\alpha_j$. In addition, since no monomial of $G$ strongly divides $x_1^{\alpha_1} \cdots x_n^{\alpha_n}$, it follows that every $m\in G$ dividing $\lcm(D)$ must be divisible by one of $x_1^{\alpha_1},\ldots,x_n^{\alpha_n}$. Hence, $\odom(S/M)\geq \# D=n$. By Theorem \ref{Theorem 3.4} and Hilbert's Syzygy Theorem [Ei], $\odom(S/M)\leq \pd(S/M)\leq n$.\\
(iv)$\Rightarrow$(i) It follows from Theorem \ref{Theorem 3.4} and Hilbert's Syzygy Theorem [Ei].\\
(iv)$\Rightarrow$(ii) It follows from Theorem \ref{Theorem 3.4}.\\
(ii)$\Rightarrow$(i) Obvious.\\
(iv)$\Rightarrow$(v) It follows from Theorem \ref{nets 3}.\\
(v)$\Rightarrow$(iv) By Corollary \ref{nets 4}, $\odom(S/M)\geq n$. By Theorem \ref{Theorem 3.4} and Hilbert's Syzygy theorem [Ei], $\odom(S/M)\leq \pd(S/M)\leq n$.
\end{proof}

\begin{theorem}\label{nets 6}
Let $M$ be a monomial ideal in $S$. Then:
\begin{enumerate}[(i)]
\item $\pd(S/M)\geq \odom(S/M)\geq \codim(S/M)$.
\item Let $\codim(S/M)=p$ and $\odom(S/M)=q$. If $q > p$ then \[ \sum\limits_{i=0}^n \betti_i(S/M)\geq 2^q > 2^p+2^{p-1}.\] 
\end{enumerate}
\end{theorem}

\begin{proof}
(i) By Theorem \ref{Theorem 3.4}, $\pd(S/M)\geq \odom(S/M)$. By Corollary \ref{nets 1'}, $\odom(S/M)\geq \# X$ for every minimal net $X$ of $M$. Since $\codim(S/M)$ is the smallest cardinality of a minimal net of $M$, it follows that $\odom(S/M)\geq \codim(S/M)$. \\
(ii) By Theorem \ref{Theorem 3.4}, $\betti_i(S/M)\geq {q\choose i}$. Therefore,
\[\sum\limits_{i=0}^n \betti_i(S/M)\geq 2^q \geq 2^{p+1}=2^p+2^p>2^p+2^{p-1}.\]
\end{proof}

In the opening paragraph of this article, we said that one of the important properties of the codimension is that it tells how short a resolution can be. In Theorem \ref{nets 6} (i), we see that the order of dominance improves this lower bound. Likewise, in a particular case, Theorem \ref{nets 6} (ii) improves an interesting lower bound found by Boocher and Seiner. In [BS], the authors prove that if $M$ is not a complete intersection, then $\sum\limits \betti_i(S/M)\geq 2^p+2^{p-1}$. Under the hypothesis that $\odom(S/M) > \codim(S/M)$, Theorem \ref{nets 6} (ii), sharpens their lower bound. 

As a marginal note, we highlight the fact that $\odom(S/M) > \codim(S/M)$ if and only if $M_{\pol}$ has two minimal nets of different cardinality. It is also worth mentioning that the inequality $\odom(S/M) > \codim(S/M)$ implies that $M$ is not Cohen-Macaulay, but the converse is false. Indeed, if $M=(ab, cd, ac, bd)$ then  $\odom(S/M)=\codim(S/M)=2$, but $\pd(S/M)=3$. 

\begin{corollary}\label{cor.}
Let $M = (m_1,\ldots,m_q)$. 
\begin{enumerate}[(i)]
\item If $\odom(S/M)=n-1$, then $\pd(S/M)=n-1$.
\item  If $\odom(S/M)=q-1$, then $\pd(S/M)=q-1$.
\end{enumerate}
\end{corollary}

\begin{proof}
(i) By Theorem \ref{nets 6} (i), if $\odom(S/M)=n-1$, then either $\pd(S/M)=n-1$ or $\pd(S/M)=n$. But the equivalent statements (i) and (iv) of Theorem \ref{nets 5} show that  $\pd(S/M) \neq n$.\\
(ii) By Theorem \ref{nets 6} (i), $\pd(S/M)\geq \odom(S/M) = q-1$, and given that $\mathbb{T}_M$ has length $q$, either $\pd(S/M) = q-1$ or $\pd(S/M)=q$. Suppose that $\pd(S/M)=q$. According to [Al1, Corollary 4.9], $\mathbb{T}_M$ is minimal, and by Theorem \ref{Taylor}, $\odom(S/M) = q$, a contradiction. Therefore, $\pd(S/M) = q-1$. 
\end{proof}

It follows from Corollary \ref{cor.} (i) that the ideals defined in Example \ref{odom by nets} must have projective dimension $n-1$.

\begin{theorem}\label{Theorem pd odom}
Let $M$ be an arbitrary monomial ideal. Then $\pd(S/M) = 1$ if and only if $\odom(S/M) = 1$.
\end{theorem}

\begin{proof}
First, we will consider the case where $M$ is squarefree. Suppose that $\pd(S/M)\ge 2$. Then $M$ must have at least $2$ minimal generators. That is, $M = (m_1,\ldots,m_q)$, with $q\ge 2$. Let 
$m = \gcd(m_1,\ldots,m_q) $, and let $M' = (m'_1,\ldots,m'_q)$, with $m'_i = \dfrac{m_i}{m}$. Notice that $\gcd(m'_1,\ldots,m'_q) = 1$, and hence, $\codim(S/M')\ge 2$. 

Let $X$ be a minimal net of $M'$. Then $\#X\ge 2$. Since every $m'_i$ is divisible by some variable of $X$, so is every $m_i$. Thus, $X$ is a net of $M$. Suppose, by means of contradiction, that $X$ is not a minimal net of $M$. Then, there is a proper subset $X'$ of $X$ that is a net of $M$ but not of $M'$. Therefore, there are a variable $x\in X'$ and an integer $i$, such that $x\mid m_i = m m'_i$ and 
$x\nmid m'_i$. This implies that $x\mid m$ and, since $M$ is squarefree, $x$ does not divide any of $m'_1,\ldots,m'_q$. Therefore, $X\setminus$\{x\} is a net of $M'$, a contradiction. It follows that $X$ is a minimal net of $M$ and, by Corollary \ref{nets 1'}, $\odom(S/M)\ge \#X\ge 2$. 

We have proven that if $\pd(S/M)\ne 1$, then $\odom(S/M)\ne 1$. On the other hand, it follows from Theorem \ref{nets 6} that if $\pd(S/M) = 1$ then $\odom(S/M) = 1$, which proves the theorem in the squarefree case. Now, the general case follows from the fact that $\odom(S/M) = \odom(S/M_{\pol})$ and $\pd(S/M) = \pd(S/M_{\pol})$.
\end{proof}

As Theorems \ref{ScarfPD}, \ref{nets 5}, \ref{Theorem pd odom}, and Corollary \ref{cor.} show, $\odom(S/M)$ is a good tool to determine $\pd(S/M)$. This is a property that distinguishes $\odom(S/M)$ from $\codim(S/M)$. In fact, Theorem \ref{Theorem pd odom} would be far from true if we replaced $\odom(S/M)$ with $\codim(S/M)$. For example, if $M=(x_1^2, x_1x_2,x_1x_3,\ldots,x_1x_n)$, then $\codim(S/M)=1$, but $\pd(S/M)=n$. The next result uses this attribute of $\odom(S/M)$ to describe monomial ideals in 3 variables.

\begin{corollary} \label{last}
Let $M$ be a monomial ideal of $S = k[x_1,x_2,x_3]$. Then, 
\begin{enumerate}[(i)]
\item $\pd(S/M) = \odom(S/M)$.
\item $M$ is Cohen-Macaulay if and only if all minimal nets of $M_{\pol}$ have the same cardinality.
\end{enumerate}
\end{corollary}
\begin{proof}
(i) By Theorem \ref{nets 5}, $\pd(S/M) = 3$ if and only if $\odom(S/M) = 3$. By Theorem \ref{Theorem pd odom}, $\pd(S/M) = 1$ if and only if $\odom(S/M) = 1$. By the preceding two statements, $\pd(S/M)\ne 1,3$ if and only if $\odom(S/M)\ne 1,3$, but this is equivalent to saying that $\pd(S/M) = 2$ if and only if $\odom(S/M) = 2$.\\
(ii) By part (i), $M$ is Cohen Macaulay if and only if $\codim(S/M) = \odom(S/M)$. Now, the theorem follows from the fact that $\codim(S/M)$ and $\odom(S/M)$ are the shortest and largest cardinalities of minimal nets of $M_{\pol}$, respectively.
\end{proof}

\section{Final remarks}
Note that Definition \ref{comb.def.} and Corollary \ref{nets 4} give two combinatorial interpretations of the notion of order of dominance. Below we reinterpret this invariant from an algebraic perspective.

It follows from its definition, that a net of $M$ is precisely the minimal generating set of a monomial prime ideal containing $M$. In particular, a minimal net is the minimal generating set of an associated prime of $M$. Since $\odom(S/M)$ is the largest cardinality of a minimal net of $M_{\pol}$, we can redefine the order of dominance of $S/M$ as the codimension of the largest associated prime of $M_{\pol}$ (to the expert, $\odom(S/M)$ is simply the ``big height" of $M_{\pol}$).

From this perspective, the proof of Theorem \ref{Theorem 3.4} follows by noting that polarization preserves Betti numbers and, localizing at this largest associated prime can only decrease Betti numbers. After localizing, we obtain a complete intersection, which is therefore minimally resolved by the Koszul complex. The ranks of the free modules in the Koszul complex provide the lower bounds for Theorem \ref{Theorem 3.4}.

The simple but powerful algebraic tools used in this last proof of Theorem \ref{Theorem 3.4} may have the reader wonder why the entirety of this work was based on the combinatorial approach to the notion of order of dominance. The short answer is that the combinatorial approach yields more information than the algebraic one. In fact, the proof of Theorem \ref{Theorem 3.4} relies on Lemma \ref{lemma 2}, a result that describes the Taylor symbols of the minimal resolution of $S/M$. Thus, Lemma \ref{lemma 2} is interesting in itself, but the algebraic perspective would have not led to such result, for once we polarize and localize, we lose track of the Taylor symbols of $S/M$. (Theorems \ref{Taylor}, \ref{Theorem pd odom},  Corollary \ref{cor.}, and the equivalence between parts (iii) and (iv) of Theorem \ref{nets 5} strongly rely on the combinatorial approach as well.)

\bigskip

\noindent \textbf{Acknowledgements}: Many thanks to an anonymous referee for his or her excellent comments that helped to improve the presentation of this paper. I am also grateful to my dear wife Danisa for her support and encouragement, and for typing this article.

\end{document}